\documentclass{amsart}

\usepackage{amsmath}
\usepackage{amssymb}

\usepackage{mathtools}      
\usepackage{graphicx}
\usepackage{color}
\usepackage{epstopdf}
\usepackage{mathptmx}      

\usepackage{constants}
\usepackage{verbatim}
\usepackage{colonequals}
\usepackage{subfigure}
\usepackage[colorinlistoftodos]{todonotes}

\usepackage{enumitem}

\def\RR{{\mathbb{R}}}
\def\NN{{\mathbb{N}}}

\DeclareMathOperator*{\argmin}{arg\,min}

\newcommand{\Energy}{\mathfrak{J}}

\theoremstyle{plain}
\newtheorem{theorem}{Theorem}[section]

\theoremstyle{definition}
\newtheorem{definition}[theorem]{Definition}
\newtheorem*{example}{Example}
\newtheorem{condition}{Condition}[section]

\theoremstyle{remark}

\let\oldtocsection=\tocsection

\let\oldtocsubsection=\tocsubsection

\let\oldtocsubsubsection=\tocsubsubsection

\renewcommand{\tocsection}[2]{\hspace{0em}\oldtocsection{#1}{#2}}
\renewcommand{\tocsubsection}[2]{\hspace{1em}\oldtocsubsection{#1}{#2}}
\renewcommand{\tocsubsubsection}[2]{\hspace{2em}\oldtocsubsubsection{#1}{#2}}

\makeatletter
\providecommand\@dotsep{5}
\renewcommand{\listoftodos}[1][\@todonotes@todolistname]{%
	\@starttoc{tdo}{#1}}
\makeatother

\begin{document}

\title{The Aubin--Nitsche Trick for Semilinear Problems}

\author[Hardering]{Hanne Hardering}
\address{Hanne Hardering\\
Technische Universit\"at Dresden\\
Institut f\"ur Numerische Mathematik\\
D-01062 Dresden\\
Germany}
\email{hanne.hardering@tu-dresden.de}

\maketitle


\begin{abstract}
The Aubin--Nitsche trick is a common tool to show $L^2$-error estimates for discretizations of $H^1$-elliptic linear partial differential equations arising for example as Euler--Lagrange equations of a quadratic energy functional \cite{ciarlet}.
The technique itself is linear: for quasilinear problems it is not applicable. We generalize the Aubin--Nitsche trick to a class of minimization problems closely related to semi-linear partial differential equations.
\end{abstract}

In textbooks on the numerical analysis of partial differential equations \cite{ciarlet,braess}, the Aubin--Nitsche trick is usually presented after establishing discretization error bounds in $O(h^{m})$ for the minimization of $H^1$-elliptic energies by $m$-th order Lagrangian finite elements, where $h$ is a mesh width parameter.
While by Poincar\'e's inequality one automatically obtains $L^2$-error estimates in $O(h^{m})$, this is not optimal as the $L^2$-interpolation error is of the better order $O(h^{m+1})$. The Aubin--Nitsche trick is then introduced as a tool to obtain optimal $L^2$-error estimates from the $H^1$-error estimates under mild additional regularity assumptions.

In these arguments the energy is quadratic, i.e., of the type $\Energy(v)=\frac{1}{2}a(v,v)+f(v)$, where $a(\cdot,\cdot):H\times H\to \RR$ denotes an $H$-elliptic scalar product on some subspace $H\subset H^1(\Omega)$, and $f\in H'$ is a linear map. The $H^1$-error estimates can be generalized to nonlinear problems \cite[Ch.\,5]{ciarlet}. The Aubin--Nitsche trick, however, relies on the linear concept of Galerkin orthogonality. To obtain optimal $L^2$-error estimates for nonlinear problems one option is to deform the nonlinear problem to a linear one and use a method of continuity argument \cite{Rannacher}. In the case of only mildly nonlinear problems, in particular semi-linear ones, this technique is not needed. Instead, we propose a new proof that replaces addition by integration in the nonlinear setting and estimate additional terms.

Linearity and semi-linearity are concepts that refer to the Euler--Lagrange equations associated to energy problems obtained by setting the first variation of the energy to zero. In this context, it is more feasible to work with properties of the energy directly. The concept of ``mildly nonlinear'' we will use is a bound on the third variation of the energy. We will call such energies predominantly quadratic. One example is $\Energy(v)\colonequals \int_{\Omega}|Dv(x)|^{2}\;dx+ \int_{\Omega}\psi(u(x))\;dx$, where $\psi$ denotes a nonlinear potential.

\tableofcontents

\section{The Aubin--Nitsche Trick for Quadratic Energies}

We will briefly summarize the basic tools we need from standard theory and the recall the Aubin--Nitsche-Trick for quadratic energies.\\
In the following $\Omega\subset \RR^{d}$ will denote an open subset with piecewise Lip\-schitz boundary $\partial\Omega$. Further $W^{k,p}(\Omega,\RR^{n})$ will denote the standard Sobolev space with the usual abbreviation $H^{k}(\Omega,\RR^{n})\colonequals W^{k,2}(\Omega,\RR^{n})$ \cite{wloka}.

\subsection{$H^{1}$-Ellipticity and $H^{1}$-Discretization Error Bounds}
We consider the minimization of energies $\Energy$ in $H\subset W_{\phi}^{1,2}(\Omega,\RR^{n})$,
where $\phi$ denotes suitable boundary data: 
\begin{align}\label{eq:Pcont}
u\in H:\qquad \Energy(u)\leq \Energy(v)\qquad \forall v\in H.
\end{align}
To bound the error of discrete approximations to minimizers of $\Energy$, we need the concept of $W^{1,2}$-ellipticity.
\begin{definition}\label{def:ellipticEnergy}
	Let $\Energy:H \to \RR$ be twice continuously differentiable, and let $\delta^{2}\Energy$ denote the second variation of $\Energy$.
	We say that $\Energy$ is
	\begin{enumerate}[label=(\alph*)]
		\item \label{def:coercive} 
		$W^{1,2}$-coercive, if
		there exists a constant $\lambda>0$ such that for all $v\in H$ and $V\in W_{0}^{1,2}(\Omega, \RR^{n})$ we have
		\begin{align}\label{eq:ellipticbelow}
		\lambda\|V\|^2_{W^{1,2}}\leq \delta^{2}\Energy(v)(V,V),
		\end{align}
		\item \label{def:bounded} 
		$W^{1,2}$-bounded, if there exists a constant $\Lambda>0$ such that for all $v\in H$ and
		for all $V,W \in W_{0}^{1,2}(\Omega,\RR^{n})$
		we have
		\begin{align}\label{eq:ellipticabove}
		\left|\delta^{2}\Energy(v)(V,W)\right|\leq \Lambda\;\|V\|_{W^{1,2}}\|W\|_{W^{1,2}},
		\end{align}
		\item  $W^{1,2}$-elliptic, if \ref{def:coercive} and \ref{def:bounded} hold.
	\end{enumerate}
\end{definition}
In order to obtain a finite-dimensional approximation of $H$, we assume that we have a conforming grid $G$ on $\Omega$, i.e., a partition into polytopes such that the closures intersect in common faces. 
\begin{definition}\label{def:widthh}
	We say that a conforming grid $G$ for the domain 
	$\Omega\subset 
	\RR^{d}$ is of width $h$ and order $m$, if for each element $T_{h}$ of $G$ there exists a 
	$C^{\infty}$-diffeomorphism $F_{h}:T_{h}\to T$ to a reference element  $T\subset \RR^{d}$ that scales with $h$ of order $m$, i.e.,
	\begin{align*}
	c\;h^{-d}\leq \|\det (DF_{h})\|_{L^{\infty}}&\leq C\; h^{-d}, & \|\partial_{\alpha}F_{h}\|_{L^{\infty}}&\leq C\; h \quad \forall \alpha=1,\ldots,d,
	\end{align*}
	and
	\begin{align*}
	|F^{-1}|_{W^{k,\infty}}\leq C\; h^{k}\quad \forall k=0,\ldots,m.
	\end{align*}
\end{definition}
Let $S^{m}_{h;\phi}\subset H\cap C(\Omega,\RR^{n})$ be a finite-dimensional approximation space for a grid $G$ on $\Omega$ of width $h$ and order $m$.
Note that this requires that the boundary data $\phi$ can be represented exactly in $S_{h;\phi}^{m}$. This requirement may be waived and replaced by a standard approximation argument for boundary data \cite[Ch. 4]{ciarlet}.

Consider the discrete approximation of~\eqref{eq:Pcont}
\begin{align}\label{eq:Pdisc}
u_{h}\in S^{m}_{h;\phi}:\qquad \Energy(u_{h})\leq \Energy(v_{h})\qquad \forall v_{h}\in S_{h;\phi}^{m}.
\end{align}
In order to control the error between $u$ and $u_{h}$, we need standard approximation conditions for the discrete space $S_{h;\phi}^{m}$.\\
The first condition consists of an estimate for the best approximation error in $S_{h;\phi}^{m}$ \cite{ciarlet}.
\begin{condition}\label{cond:1}
	Let $kp>d$, $m\geq k-1$, and $u\in W^{k,p}_{\phi}(\Omega,\RR^{n})$. For small enough $h$ let there exist a map $u_{I}\in S_{h;\phi}^{m}$ and constants $\Cl{c:m+1Deriv}, \Cl{c:cond1} $ with 
	\begin{align}\label{eq:cond1b}
	|u_{I}|_{W^{l,q}(\Omega,\RR^{n})}\leq \Cr{c:m+1Deriv}\;	|u|_{W^{l,q}(\Omega,\RR^{n})}
	\end{align}
	for all $k-\frac{d}{p}\leq l\leq k$ and $q\leq \frac{pd}{d-p(k-l)}$ such that $u_{I}$ fulfills
	on each element $T_{h}\in G$ the estimate
	\begin{align}\label{eq:cond1a}
	\|u-u_{I}\|_{L^{p}} + h\;|u-u_{I}|_{W^{1,p}} \leq
	\Cr{c:cond1}\;h^{k}\;|u|_{W^{k,p}}.
	\end{align}
\end{condition}
The second condition is generally known as an inverse estimate.
\begin{condition}\label{cond:inv}
	On a grid $G$ of width $h$ and order $m$, under the additional assumption that $F^{-1}_{h}:T\to T_{h}$ scales with order $2$ for every $T_{h}\in G$, let there exist a constant $\Cl{c:inverseEst}$, such that		
	\begin{align}\label{eq:inverse}
	\|v_{h}\|_{W^{1,p}(T_{h},\RR^{n})}\leq \Cr{c:inverseEst}\;h^{-d\max\{0,\frac{1}{q}-\frac{1}{p}\}}\|v_{h}\|_{W^{1,q}(T_{h},\RR^{n})}
	\end{align}
	for any $v_{h}\in S^{m}_{h}(T_{h},M)$ and for any $p,q\in[1,\infty]$.
\end{condition}
Note that the discrete functions in $S_{h;\phi}^{m}$ are globally only of $C^{0}\cap W^{1,2}$-smoothness.
Whenever we consider higher Sobolev norms, we implicitly define them as grid dependent, i.e.,
	\begin{align}
	|u|_{W^{k,p}(\Omega,\RR^{n})}\colonequals \left(\sum_{T_{h}\in G} |u|_{W^{k,p}(T_{h},\RR^{n})}^{p}\right)^{\frac{1}{p}}.
	\end{align}
By summation over the elements of $G$, estimates like~\eqref{eq:cond1a} and~\eqref{eq:inverse} carry over to global grid-dependent norms.
For $H^{1}$-elliptic energies, C\'ea's Lemma with Condition~\ref{cond:1} yields the following $W^{1,2}$-error estimate \cite{ciarlet}.
\begin{theorem}\label{T:H1err}
	Let $2(m+1)>d$, and $m\geq 1$. Assume that $u\in W_{\phi}^{m+1,2}(\Omega,M)$ is a minimizer of an $H^{1}$-elliptic $\Energy:H\to \RR$.
	Then the discrete minimizer
	\begin{align*}
	u_{h}\colonequals \argmin_{v_{h}\in S_{h}^{m}} \Energy(v_{h})
	\end{align*}
	fulfills the a priori error estimate
	\begin{align}\label{eq:H1err}
	\|u-u_{h}\|_{W^{1,2}}\leq \Cl{c:H1err} h^{m}|u|_{W^{m+1,2}}.
	\end{align}
\end{theorem}

\subsection{The Aubin--Nitsche Trick For Quadratic Energies}\label{sec:lin}
The purpose of the Aubin--Nitsche trick is to show that for $W^{1,2}$-elliptic minimization problems the $L^{2}$-discretization error is in $O(h^{m+1})$. 

We recall the Aubin--Nitsche lemma for the approximation of a quadratic minimization problem in $H=H^{1}_{0}(\Omega,\RR)$ 
by Lagrangian finite elements.
For an elliptic bilinear form $a(\cdot,\cdot)$ and given $f\in H^{-1}$
consider the energy $J(v)=\frac{1}{2}a(v,v) - (f,v)$, the variational equalities
\begin{alignat}{5}
 u&\in H: &\quad a(u,v)&=(f,v) \quad &\forall v&\in H,\label{eq:linP}\\
 u_{h}&\in S_{h;0}^{m}: &\quad a(u_{h},v_{h})&=(f,v_{h}) \quad &\forall v_{h}&\in S_{h;0}^{m},\label{eq:linPh}
\end{alignat}
and the adjoint problem
\begin{align}\label{eq:adjproblin}
w\in H:\qquad a(v,w)&=(g,v)\qquad \forall v\in H ,
\end{align}
where $g\colonequals u-u_{h}$.
We assume $H^2$-regularity of the adjoint problem, i.e., $|w|_{H^2}\leq C\|g\|_{L^2}$.
The subtraction of equations~\eqref{eq:linP} and~\eqref{eq:linPh} with the same test function $v_{h}\in S_{h;0}^{m}\subset H$ yields the concept of Galerkin orthogonality, i.e.,
\begin{align}\label{eq:galerkinLin}
a(u-u_{h},v_{h})=0\qquad \forall v_{h}\in S_{h;0}^{m}.
\end{align}
Using Galerkin orthogonality and the $H^1$-ellipticity of $a(\cdot,\cdot)$, we can then estimate
\begin{align*}
\|u-u_{h}\|^{2}_{L^2}&=(g,u-u_{h})=a(u-u_{h},w)= a(u-u_{h},w-w_{I})\\
&\leq \Lambda\|u-u_{h}\|_{H^{1}}\|w-w_{I}\|_{H^{1}}\\
&\leq C h^{m} |u|_{H^{k}}\; h\;|w|_{H^{2}}\\
&\leq C h^{m+1} |u|_{H^2} \|u-u_{h}\|_{L^2}.
\end{align*}
Galerkin orthogonality $a(u-u_{h},w_{I})=0$ is the essential tool used here. It allows to incorporate an approximation of $w$ and thus leads to a better order estimate than the $H^1$-error. It is at first glance a purely linear concept that is verified by adding the equations~\eqref{eq:linP} and~\eqref{eq:linPh} for the same discrete test function, a technique that does not work for nonlinear energies. We will circumvent the need for addition by integration and show that semi-linearity of the Euler--Lagrange equation will then be sufficient to obtain equivalent error estimates.

\section{Semi-linearity and Predominantly Quadratic Energies}
We now introduce the concept of predominantly quadratic energies. Energies with semi-linear Euler--Lagrange equations fall in this category. At the same time, this property is exactly what we need for the $L^2$-error bounds.
Let the energy functional $\Energy:W^{1,2}(\Omega,\RR^{n})\to \RR$ be given by
\begin{align*}
\Energy(v)\colonequals \int_{\Omega}L(Dv,v,x)\;dx,
\end{align*}
where  $L:\RR^{n\times d}\times \RR^{n}\times \Omega\to \RR$, $(p,z,x)\mapsto L(p,z,x)$, is a smooth Lagrangian, and $D$ denotes the (weak) differentiation operator of a function from $\Omega\subset \RR^{d}$ to $\RR^{n}$.

We calculate the first variation of $\Energy$ at a function $u\in W^{1,2}(\Omega,\RR^{n}) $ in direction $V\in W^{1,2}_{0}(\Omega,\RR^{n})$:
\begin{align*}
	\delta\Energy(u)(V)&=\frac{d}{d\tau}\Big|_{\tau=0}\Energy(u+\tau V)\\
	&= \int_{\Omega}\partial_{p}L(Du,u,x)\cdot DV + \partial_{z}L(Du,u,x)\cdot V\;dx\\
	&=\int_{\Omega}\left(- D_{x} \partial_{p}L(Du,u,x)+ \partial_{z}L(Du,u,x)\right)\cdot V \;dx , 
\end{align*}
where $\partial_{p}, \partial_{z}, \partial_{x}$ denote partial differentiation of $L$ with respect to the corresponding variables.
Setting the variation to zero yields the corresponding system of Euler--Lagrange equations (cf. \cite{Evans})
\begin{align}\label{eq:eulerlagrange}
-\sum_{i=1}^{d}\frac{d}{dx_{i}}\left(\partial_{p_{i}^{k}}L(Du,u,x)\right) + \partial_{z^{k}}L(Du,u,x)=0\qquad \textrm{in}\ \Omega\ (k=1,\ldots,n).
\end{align}
On the other hand a semi-linear system of partial differential equations
\begin{align*}
a_{ij}(x)D^{ij}u(x) + a_{0}(Du(x),u(x), x)= 0
\end{align*}
is characterized by the independence of the coefficients $a_{ij}$ of the solution $u$.
Thus, the Euler-Lagrange equation \eqref{eq:eulerlagrange} is semi-linear if the operator $\partial_{p}^{2}L$ is independent of $u$, i.e.,
\begin{align*}
	\partial_{p}^{2}L(Du,u,x)=\partial_{p}^{2}L(x), 
\end{align*}
and accordingly
\begin{align*}
\partial_{p}^{3}L(Du,u,x)=0, \quad\textrm{and}\quad \partial_{z} \partial_{p}^{2}L(Du,u,x)=0.
\end{align*}
These vanishing third order derivatives appear in the third variation of $\Energy$.
In general, the second variation of $\Energy$ reads
\begin{multline*}
\delta^{2}\Energy(u)(V,W)
= \int_{\Omega}\partial_{p}^{2}L(Du,u,x)(DV,DW) +  \partial_{z}\partial_{p}L(Du,u,x)(DV,W)\\
 + \partial_{p}\partial_{z}L(Du,u,x)(V,DW) + \partial_{z}^{2}L(Du,u,x)(V,W)\;dx,
\end{multline*}
and the third variation is
\begin{align*}
\delta^{3}\Energy(u)(V,W,U)
&=\int_{\Omega}\partial_{p}^{3}L(Du,u,x)(DV,DW,DU) +\partial_{z}\partial_{p}^{2}L(Du,u,x)(DV,DW,U)\\
&\quad  +  \partial_{p}\partial_{z}\partial_{p}L(Du,u,x)(DV,W,DU)+\partial_{z}^{2}\partial_{p}L(Du,u,x)(DV,W,U)\\
&\quad+ \partial^{2}_{p}\partial_{z}L(Du,u,x)(V,DW,DU)+ \partial_{z}\partial_{p}\partial_{z}L(Du,u,x)(V,DW,U)\\
&\quad + \partial_{p}\partial_{z}^{2}L(Du,u,x)(V,W,DU) + \partial_{z}^{3}L(Du,u,x)(V,W,U)\;dx.
\end{align*}
If \eqref{eq:eulerlagrange} is semi-linear, the third variation reduces to
\begin{align}\label{eq:thirdvarsemilin}
\delta^{3}\Energy(u)(V,W,U)
&=\int_{\Omega} \partial_{z}^{2}\partial_{p}L(Du,u,x)(DV,W,U) + \partial_{z}\partial_{p}\partial_{z}L(Du,u,x)(V,DW,U)\\
&\quad + \partial_{p}\partial_{z}^{2}L(Du,u,x)(V,W,DU) + \partial_{z}^{3}L(Du,u,x)(V,W,U)\;dx.\notag
\end{align}
Thus, semi-linearity of the Euler--Lagrange equation necessarily implies the dependence of each term of the third variation on at most one direction gradient.

We turn this observation into a definition.
More generally, we will consider predominantly quadratic energies. By this we mean the following:

\begin{definition}\label{def:almostLinear}
	Let $q>\max\{d,2\}$ and $\Energy:H \to \RR$ be an energy functional.
	We say that $\Energy$ is predominantly quadratic with respect to $q$ if $\Energy$ is $C^{3}$, 
	and for any $u\in H \cap W^{1,q}(\Omega,\RR^{n})$, $U\in W^{2,2}(\Omega,\RR^{n})$, and $V\in W^{1,2}\cap W^{o,r}(\Omega,\RR^{n})$ with either $(o,r)=(1,2)$, or $o=0$ and $r\leq d$, there exists a constant $\Cl{c:almostLinear}$ possibly depending on $\|u\|_{W^{1,q}}$ such that
	\begin{align}\label{eq:almostLinear}
		|\delta^{3}\Energy(u)(U,V,V)|\leq \Cr{c:almostLinear}\|U\|_{W^{2,2}}\|V\|_{W^{1,2}}
		\|V\|_{W^{o,r}}.
	\end{align}
\end{definition}
\begin{example}
	We have seen in \eqref{eq:thirdvarsemilin} that as long as the Lagrangian $L$ is smooth enough and its third variations are bounded in $L^{q}$ in terms of $\|u\|_{W^{1,q}}$, 
	the leading term of the third variation of the corresponding energy will have a bound of the form
	\begin{align*}
		|\delta^{3}\Energy(u)(U,V,V)|\leq C\;\left(\int_{\Omega}\left(|DV||V||U| + |V|^{2}|DU|\right)^{\frac{q}{q-1}}\;dx\right)^{1-\frac{1}{q}},
	\end{align*}
	if we assume semi-linearity of the Euler--Lagrange system.
	Thus, such an energy is predominantly quadratic (H\"older's inequality).
	\end{example}
	
	\begin{example}	
	The leading term of the third variation of the energy for a typical quasi-linear equation, 
	e.g., the minimal surface energy for graphs $\Energy(u)=\int_{\Omega}\sqrt{1+|Du|^{2}}\;dx$, has the form
	\begin{align*}
		|\delta^{3}\Energy(v)(U,V,V)|\leq C\;\left(\int_{\Omega}\left(|U||D V|^{2}+|D U||D V||V|\right)^{\frac{q}{q-1}}\;dx\right)^{1-\frac{1}{q}}.
	\end{align*}
	For $d< 4$ such an energy is predominantly quadratic with respect to $q=\infty$, but not in general.
	\end{example}

\section{Galerkin Orthogonality and the Adjoint Problem}
%
We consider the variational formulations of the problems \eqref{eq:Pcont} and \eqref{eq:Pdisc}
\begin{alignat}{5}
u&\in H:&\qquad \delta\Energy(u)(V)&=0\qquad &\forall V&\in W_{0}^{1,2}(\Omega,\RR^{n}),\label{eq:PcontVar}\\
u_{h}&\in S_{h;\phi}^{m}:&\qquad \delta\Energy(u_{h})(V_{h})&=0\qquad &\forall V_{h}&\in S_{h;0}^{m}. \label{eq:PdiscVar}
\end{alignat}
These correspond to~\eqref{eq:linP} and~\eqref{eq:linPh} in the linear setting. 
By inserting a discrete test function into~\eqref{eq:PcontVar} we obtain by the fundamental theorem of calculus (replacing subtraction in the linear setting)
\begin{align}\label{eq:galerkin}
0=  \delta\Energy(u_{h})(V_{h})- \delta\Energy(u)(V_{h})=\int_{0}^{1}\delta^{2}\Energy(\Gamma(t))(V_{h},u_{h}-u)\;dt,
\end{align}
where $\Gamma(t)=(1-t)u+tu_{h}$. This is a nonlinear generalization of Galerkin orthogonality. Note that for a quadratic energy $\delta^{2}\Energy$ is independent of the function $\Gamma(t)$ and we recover the standard notion of Galerkin orthogonality~\eqref{eq:galerkinLin}. 

We now define a nonlinear generalization of the adjoint problem \eqref{eq:adjproblin} featuring in the Aubin--Nitsche trick.
For nonlinear energies the adjoint problem is essentially a linearization of problem~\eqref{eq:PcontVar} \cite{Rannacher}
with a right hand side that is given by the difference of the solutions $u$ and $u_{h}$ to~\eqref{eq:PcontVar} and~\eqref{eq:PdiscVar}, respectively:
\begin{align}\label{eq:deformedContVec}
W\in W_{0}^{1,2}(\Omega,\RR^{n})\ :\ \delta^{2}\Energy(u)(W,V) =-(V,u_{h}-u)_{L^{2}}\qquad \forall V\in W_{0}^{1,2}(\Omega,\RR^{n}).
\end{align}
Note that as long as the operator $\partial_{p}^{2}L(x)$ is in $W^{1,q}$, $\partial_{p}\partial_{z}L(Du,u,x)$ is in $L^{q}$, $\partial_{z}^{2}L(Du,u,x)$ is in $L^{\frac{\max\{q,4\}}{2}}$, and $u_{h}-u$ is in $L^{2}$ for $q>\max\{2,d\}$, standard regularity results for linear elliptic systems \cite{lady} yield that the adjoint problem is $H^{2}$-regular, i.e., that the solution $W$ fulfills
\begin{align}\label{eq:H2reg}
\|W\|_{W^{2,2}(\Omega,\RR^{n})}\leq C\;\|u-u_{h}\|_{L^{2}(\Omega,\RR^{n})}.
\end{align}
\section{$L^{2}$-Error Estimate}
We will now combine the nonlinear Galerkin orthogonality~\eqref{eq:galerkin} with the standard estimate for $H^1$-elliptic energies~\eqref{eq:H1err} to show that a higher order estimate for the $L^2$-error for predominantly quadratic energies can be obtained by the Aubin--Nitsche trick analogous to the linear setting described in Section~\ref{sec:lin}.
\begin{theorem}\label{T:L2err}
Let $m\in \NN$, and $2(m+1)>d$.
Assume that $u\in W_{\phi}^{m+1,2}(\Omega,\RR^{n})$ is a minimizer of an elliptic energy $\Energy$ that is predominantly quadratic with respect to $q>\max\{2,d\}$ with $q\leq 2d$ if $d-2m=1$.
Let $u_{h}$ be a (local) minimizer of $\Energy$ in $S_{h;\phi}^{m}$ fulfilling~\eqref{eq:H1err}.
Finally, suppose that the adjoint problem~\eqref{eq:deformedContVec} is $H^{2}$-regular, i.e., that its solution $W$ fulfills \eqref{eq:H2reg}.
	Then there exists a constant $\Cl{c:L2const}$, such that
	\begin{align}
	\|u-u_{h}\|_{L^{2}(\Omega,\RR^{n})}\leq \Cr{c:L2const}\;h^{m+1},
	\end{align}
	where $\Cr{c:L2const}$ depends nonlinearily on $\|u\|_{W^{1,q}}$ and $|u|_{W^{m+1,2}}$.
\end{theorem}
\begin{proof}
	We insert $V\colonequals u_{h}-u$ into \eqref{eq:deformedContVec}, and obtain
	\begin{align*}
	\|u-u_{h}\|_{L^{2}}^{2}&= -\delta^{2}\Energy(u)(W,u_{h}-u),
	\end{align*}
	where $W\in W_{0}^{2,2}(\Omega,\RR^{n})$ is the solution of \eqref{eq:deformedContVec}.
	Let $W_{I}\in S_{h;0}^{m}$ be an approximation of $W$ in the sense of Condition~\ref{cond:1}.
	As $u_{h}$ is a local minimizer in $S_{h;\phi}^{m}$, generalized Galerkin orthogonality \eqref{eq:galerkin} holds, so that for $\Gamma(t)=(1-t)u+tu_{h}$
	\begin{align}
		\|u_{h}-u\|_{L^{2}(\Omega,\RR^{n})}^{2} &=  - \delta^{2}\Energy(u)(W,u_{h}-u) + \int_{0}^{1} \delta^{2} \Energy(\Gamma(t))\left(W_{I},u_{h}-u \right)\;dt \nonumber \\
	&=  \int_{0}^{1}\int_{0}^{t} \frac{1}{t}\frac{d}{ds} \delta^{2} \Energy(\Gamma(s))\left(s W_{I} + \left(t-s\right)W,u_{h}-u\right)\;ds\;dt \nonumber \\
	&=  \int_{0}^{1}\int_{0}^{t} \frac{1}{t} \delta^{3} \Energy(\Gamma(s))\left(s W_{I}+\left(t-s\right)W,u_{h}-u,u_{h}-u\right)\;ds\;dt \nonumber \\
	&\qquad  +  \int_{0}^{1}\int_{0}^{t} \frac{1}{t} \delta^{2} \Energy(\Gamma(s))(W_{I}- W,u_{h}-u)\;ds\;dt. \label{eq:dummieGalerkin}
	\end{align}
	The second integral in \eqref{eq:dummieGalerkin}
	is estimated using the ellipticity assumption~\eqref{eq:ellipticabove}:
	\begin{align*}
	 \int_{0}^{1}\int_{0}^{t} \frac{1}{t} \delta^{2} \Energy(\Gamma(s))(W_{I}- W,u_{h}-u)\;ds\;dt
	\leq \Lambda \|W_{I}- W\|_{W^{1,2}(\Omega,\RR^{n})} \|u_{h}-u\|_{W^{1,2}(\Omega,\RR^{n})}.
	\end{align*}
	Using Condition~\ref{cond:1} on $W_{I}$, the $H^{1}$-error bound \eqref{eq:H1err}, and the $H^{2}$-regularity  \eqref{eq:H2reg}, we obtain
	\begin{align*}
	 \int_{0}^{1}\int_{0}^{t} \frac{1}{t} \delta^{2} \Energy(\Gamma(s))(W_{I}- W,u_{h}-u)\;ds\;dt
	 &\leq C\;h^{m+1} |W|_{W^{2,2}(\Omega,\RR^{n})} \|u\|_{W^{m+1,2}(\Omega,\RR^{n})}\\
	 &\leq C\;h^{m+1}  \|u\|_{W^{m+1,2}(\Omega,\RR^{n})}\|u_{h}-u\|_{L^{2}(\Omega,\RR^{n})}.
	\end{align*}
	In order to estimate the first integral term in \eqref{eq:dummieGalerkin} we use that $\Energy$ is predominantly quadratic
	\begin{multline*}
	|\delta^{3} \Energy(\Gamma(s))(s W_{I}+(t-s)W,u_{h}-u,u_{h}-u)|\\
	\begin{aligned}
	&\leq C(\|\Gamma(s)\|_{W^{1,q}}) \|u_{h}-u\|_{W^{1,2}}\|u_{h}-u\|_{W^{o,r}} \left(s \|W_{I}\|_{W^{2,2}} + (t-s)\|W\|_{W^{2,2}}\right).
	\end{aligned}
	\end{multline*}
	Using again Condition~\ref{cond:1} on $W_{I}$, the $H^{1}$-error bound, and the $H^{2}$-regularity, we obtain
	\begin{align*}
	\int_{0}^{1}\int_{0}^{t} &\frac{1}{t}  \delta^{3} \Energy(\Gamma(s))\left(s W_{I}+\left(1-s\right)W,u_{h}-u,u_{h}-u\right)\;ds\;dt \\
	&\leq C\int_{0}^{1}\int_{0}^{t}  C(\|\Gamma(s)\|_{W^{1,q}})\;ds\;dt  \|u_{h}-u\|_{W^{1,2}}\|u_{h}-u\|_{W^{o,r}} \|u_{h}-u\|_{L^{2}}\\
	& \leq C(\max\{\|u\|_{W^{1,q}},\|u_{h}\|_{W^{1,q}}\})h^{m}  |u|_{W^{m+1,2}} \|u_{h}-u\|_{W^{o,r}} \|u_{h}-u\|_{L^{2}}.
	\end{align*}
	Note that we can assume that
	\begin{align*}
	\|u_{h}\|_{W^{1,q}}\leq C\;\left(\|u\|_{W^{1,q}}+|u|_{W^{m+1,2}}\right)
	\end{align*}
	for $h$ small enough, as we can use the inverse estimate in Condition~\ref{cond:inv} on $(u_{h}-u_{I})\in S_{h;0}^{m}$, Condition~\ref{cond:1} on $u_{I}$, and the $H^1$-error bound to estimate
	\begin{align*}
	\|u_{h}\|_{W^{1,q}}&\leq \|u_{h} - u_{I}\|_{W^{1,q}}+ \|u_{I}\|_{W^{1,q}}\\
	&\leq h^{-d(\frac{1}{2}-\frac{1}{q})}\|u_{h} - u_{I}\|_{W^{1,2}}+ C\|u\|_{W^{1,q}}\\
	&\leq C h^{m-d(\frac{1}{2}-\frac{1}{q})}|u|_{W^{m+1,2}}+ C\|u\|_{W^{1,q}}.
	\end{align*}
	Note that $m-d(\frac{1}{2}-\frac{1}{q})\geq 0$, if we assume $q\leq 2d$ in the case $d-2m=1$.
	If $o=1$ and $r=2$,
	then we can use the $H^{1}$-error bound again to obtain
	\begin{multline*}
	\int_{0}^{1}\int_{0}^{t} \frac{1}{t}\int_{0}^{1}  \delta^{3} \Energy(\Gamma(s))\left(s W_{I}+\left(1-s\right)W,u_{h}-u,u_{h}-u\right)\;ds\;dt\\
	\leq 
	C(\|u\|_{W^{1,q}},|u|_{W^{m+1,2}})h^{2m} \|u_{h}-u\|_{L^{2}}
	\end{multline*}
	with $2m\geq m+1$.
	
	If instead $o=0$ and $r\leq d$, then either we are in the same situation as before, or $d\geq4$ and $\frac{2d}{d-2}\leq r\leq d$. In that case $L^{p}$-interpolation with $\epsilon=h$ yields
	\begin{align*}
	\|u_{h}-u\|_{L^{r}} &\leq h\; \|u_{h}-u\|_{L^{\infty}} + h^{1-\frac{r(d-2)}{2d}}\|u_{h}-u\|_{L^{\frac{2d}{d-2}}}\\
	&\leq  h \;\left( \|u\|_{L^{\infty}} +\|u_{h}\|_{L^{\infty}}\right) + C\; h^{m+1-\frac{r(d-2)}{2d}} |u|_{W^{m+1,2}}.
	\end{align*}
	As $(m+1)\geq \frac{d}{2}$ and $d\geq r$, we have $m+1-\frac{r(d-2)}{2d} \geq 1$.
	Thus, we obtain also for this case
	\begin{multline*}
	\int_{0}^{1}\int_{0}^{t} \frac{1}{t}\int_{0}^{1}  \delta^{3} \Energy(\Gamma(s))\left(s W_{I}+\left(1-s\right)W,u_{h}-u,u_{h}-u\right)\;ds\;dt\\
	\leq 	C(\|u\|_{W^{1,q}},|u|_{W^{m+1,2}})\;h^{m+1}\|u_{h}-u\|_{L^{2}}.
	\end{multline*}
	This yields the assertion.
\end{proof}

\bibliographystyle{alpha}     
\bibliography{literatur}

\end{document}